\documentclass[11pt]{article}
 
\usepackage[margin=1in]{geometry}
\usepackage[T1]{fontenc}
\usepackage{lmodern}
\usepackage{microtype}
\usepackage{amsmath,amssymb,amsthm,mathtools}
\usepackage{booktabs,tabularx,array,longtable}
\usepackage{enumitem}
\usepackage[hidelinks]{hyperref}
\usepackage[nameinlink,capitalise,noabbrev]{cleveref}
\usepackage{xurl}
\usepackage{mathtools}
 
\allowdisplaybreaks
\setlist[itemize]{leftmargin=2em,itemsep=2pt,topsep=4pt}
\setlist[enumerate]{leftmargin=2.2em,itemsep=2pt,topsep=4pt}
 
\hypersetup{
  pdftitle={Proper Hat-Guessing on Two-Spine Book Graphs},
  pdfauthor={Yulin Zhai},
  pdfsubject={Proper hat-guessing on two-spine book graphs},
  pdfkeywords={Hat guessing game, Proper hat guessing number, Book graph, Coverability, Pseudoforest}
}
 
\usepackage{aliascnt}

\newtheorem{theorem}{Theorem}[section]

\newaliascnt{lemma}{theorem}
\newtheorem{lemma}[lemma]{Lemma}
\aliascntresetthe{lemma}

\newaliascnt{proposition}{theorem}

\aliascntresetthe{proposition}

\newaliascnt{corollary}{theorem}
\newtheorem{corollary}[corollary]{Corollary}
\aliascntresetthe{corollary}

\newaliascnt{conjecture}{theorem}
\newtheorem{conjecture}[conjecture]{Conjecture}
\aliascntresetthe{conjecture}

\theoremstyle{definition}
\newaliascnt{definition}{theorem}
\newtheorem{definition}[definition]{Definition}
\aliascntresetthe{definition}

\newaliascnt{example}{theorem}

\aliascntresetthe{example}

\theoremstyle{remark}
\newaliascnt{remark}{theorem}
\newtheorem{remark}[remark]{Remark}
\aliascntresetthe{remark}

\newaliascnt{observation}{theorem}
\newtheorem{observation}[observation]{Observation}
\aliascntresetthe{observation}

\crefname{lemma}{Lemma}{Lemmas}
\Crefname{lemma}{Lemma}{Lemmas}
\crefname{proposition}{Proposition}{Propositions}
\Crefname{proposition}{Proposition}{Propositions}
\crefname{corollary}{Corollary}{Corollaries}
\Crefname{corollary}{Corollary}{Corollaries}
\crefname{conjecture}{Conjecture}{Conjectures}
\Crefname{conjecture}{Conjecture}{Conjectures}
\crefname{definition}{Definition}{Definitions}
\Crefname{definition}{Definition}{Definitions}
\crefname{example}{Example}{Examples}
\Crefname{example}{Example}{Examples}
\crefname{remark}{Remark}{Remarks}
\Crefname{remark}{Remark}{Remarks}
\crefname{observation}{Observation}{Observations}
\Crefname{observation}{Observation}{Observations}
 
\newcommand{\HGP}{\operatorname{HGP}}
\newcommand{\HG}{\operatorname{HG}}
\newcommand{\F}{\mathbb F}
\newcommand{\cC}{\mathcal C}
\newcommand{\cR}{\mathcal R}

\renewcommand{\vec}[1]{\boldsymbol{#1}}
\newcommand{\set}[1]{\left\{#1\right\}}
\newcommand{\abs}[1]{\left|#1\right|}
\newcommand{\restr}[1]{\mathord{\upharpoonright}_{#1}}
\DeclareMathOperator{\supp}{supp}
\DeclareMathOperator{\AGL}{AGL}
 
\title{Proper Hat-Guessing on Two-Spine Book Graphs}
 
\author{%
  Yulin Zhai\thanks{Institution email:
  \href{mailto:yz5460@columbia.edu}
  {\texttt{yz5460@columbia.edu}}.}\\[0.5em]
  {\small Department of Applied Physics and Applied Mathematics,
  Columbia University}\\[-0.15em]
  {\small 500 W.\ 120th Street, New York, NY 10027, USA.}\\[-0.15em]
}
\date{}
 
\begin{document}
 
\maketitle
 
\begin{abstract}
In the proper variant of the classical hat-guessing game on a graph, an adversary \emph{properly} colors the
vertices from a palette of \(q\) colors.  Each vertex sees only its neighbors'
colors and all vertices simultaneously guess their own color. The players win if at least one guess is correct.  We study this game on the book graph
\(B_{k,n}=K_k\vee\overline{K_n}\), with \(k\) mutually adjacent spine
vertices and \(n\) independent pages.

We first give a coverability characterization valid for every fixed spine
size.  Write \(\HGP(G)\) for the proper hat-guessing number of \(G\).
For a finite configuration \(P\), let \(\supp(P)\) denote the set of
colors appearing in its tuples.
Let \(C_k\) be the minimum of \(\abs P+\abs{\supp(P)}\) over all
non-coverable finite configurations \(P\) of proper \(k\)-tuples.  We prove
$ \sup_{n\geq 1}\HGP(B_{k,n})=C_k $ and that $\HGP(B_{k,n})=C_k\ \text{for all sufficiently large \(n\)}.$
Thus, the asymptotic problem for every fixed \(k\) reduces to a finite
extremal invariant.  In particular, coverability of two-spine configurations is equivalent to
pseudoforestness, and we determine the associated extremal problem exactly:
\(C_2=11\), with precisely two types of extremal obstruction.  Consequently,
\(\HGP(B_{2,n})\leq 11\) for every \(n\), with equality for all sufficiently
large \(n\); we give an explicit probabilistic estimate with a stabilization
threshold of at most $4 \times10^8$.
 
We also resolve the first two previously open finite cases. An explicit
seven-color construction with affine symmetry proves
$
  \HGP(B_{2,3})=7.
$
A counting-rigidity argument establishes the linear upper bound
$
  \HGP(B_{2,n})\leq n+3\ \text{for all } n\geq 4,
$
which together with monotonicity yields
$
  \HGP(B_{2,4})=7.
$
Finally, a general box
obstruction gives explicit uniform bounds on \(C_k\).
\end{abstract}
 
\medskip
\noindent\textbf{Keywords:}
Hat guessing game; Proper hat guessing number; Book graph;
Coverability; Pseudoforest
 
\smallskip
\noindent\textbf{MSC 2020:}
Primary 05C57; Secondary 05C15, 05D15, 91A43
 
\section{Introduction}
 
\subsection{Game}
 
Let \(G=(V,E)\) be a finite simple graph and let \([q]=\{0,1,\dots,q-1\}\).
Before play begins, the vertices agree on deterministic guess functions.  An
adversary then assigns a color from \([q]\) to every vertex.  Each vertex sees
the colors on its open neighborhood, and all vertices simultaneously guess
their own colors.  The team wins if at least one guess is correct.  In the
classical game, the coloring is arbitrary; in the \emph{proper} game, the
adversary must assign distinct colors to adjacent vertices.  The corresponding
thresholds are the hat-guessing number \(\HG(G)\) and the proper
hat-guessing number \(\HGP(G)\), which denotes the largest $q$ for which the team can force a win.
 
The book graph considered here is
$
   B_{k,n}=K_k\vee\overline{K_n}.
$
Its \(k\) clique vertices are the \emph{spines}, its \(n\) independent
vertices are the \emph{pages}, and every page is adjacent to every spine. Book graphs are particularly interesting in hat-guessing games because they consist of a bounded clique observed by arbitrarily many mutually nonadjacent players.
 
The proper variant was introduced in~\cite{ABBCBKMMRT}.  That work determined
several elementary families and left
\(\HGP(B_{2,3})\in\{6,7\}\); it also asked for uniform bounds on
\(\max_n\HGP(B_{k,n})\) for fixed \(k\).  The results below determine the
maximum for \(k=2\) and reduce the problem for every fixed \(k\) to a finite
coverability constant.
 
\subsection{Main results}
 
For a finite set \(P\) of proper \(k\)-tuples, let \(\supp(P)\) be the set of
colors occurring in its coordinates.  A configuration is \emph{coverable} if
its tuples can be assigned to the \(k\) spines so that no two tuples assigned
to spine \(i\) agree in any coordinates other than \(i\).  The precise
definition appears in \cref{sec:coverability}.
\begin{equation}\label{eq:Ck-intro}
 C_k \coloneqq \min\set{\abs P+\abs{\supp(P)}:
      P\text{ is finite and non-coverable}}.
\end{equation}
 
\begin{theorem}[Stabilization for fixed spine size]\label{thm:intro-stabilization}
For every \(k\geq2\),
$
   \sup_{n\geq1}\HGP(B_{k,n})=C_k.
$
Moreover, there is \(N_k<\infty\) such that
\(\HGP(B_{k,n})=C_k\) for every \(n\geq N_k\).
\end{theorem}
 
For \(k=2\), a configuration is a bipartite graph and coverability is
equivalent to every component having at most one cycle.
 
\begin{theorem}[Two spines]\label{thm:intro-two-spines}
The coverability constant is \(C_2=11\).  Consequently,
$
  \HGP(B_{2,n})\leq11\ \text{for}\ n\geq1\ \text{and} \
  \HGP(B_{2,n})=11\ \text{for every}\ n\geq n_0,
$
where
$
  n_0\leq
  \left\lceil 110(\log 2)\,9^7\right\rceil
  =364{,}683{,}163 \approx 4 \times 10^8.
$
In particular,
\[
   \lim_{n\to\infty}\HGP(B_{2,n})
   =\sup_{n\geq1}\HGP(B_{2,n})=11.
\]
\end{theorem}
 
We also determine the first two unresolved cases left by previous work and give a
linear bound that is useful before stabilization.
 
\begin{theorem}[Small exact cases]\label{thm:intro-small}
We prove $
  \HGP(B_{2,3})=\HGP(B_{2,4})=7\
  \text{and}\
  \HGP(B_{2,n})\leq n+3\ \text{for}\ n\geq4.
$
Moreover, \(\HGP(B_{2,n})\geq8\) for every \(n\geq6\).
\end{theorem}
 
For general \(k\), the constant in \cref{eq:Ck-intro} has a useful explicit
upper bound.
 
\begin{theorem}[Box bound]\label{thm:intro-box}
Let \(k\geq2\) and let \(m_1,\dots,m_k\geq2\) be integers satisfying
\(\sum_{i=1}^k1/m_i<1\).  Then
\[
  C_k\leq \prod_{i=1}^k m_i+\sum_{i=1}^k m_i,
  \qquad
  \HGP(B_{k,n})\leq
  \prod_{i=1}^k m_i+\sum_{i=1}^k m_i\
  \text{for}\  n\geq1.
\]
\end{theorem}
 
\subsection{Background and Related Work}
 
Hat-guessing games were introduced by Butler, Hajiaghayi, Kleinberg, and
Leighton~\cite{ButlerEtAl}.  Alon, Ben-Eliezer, Shangguan, and
Tamo~\cite{AlonEtAl} developed general constructions and probabilistic
obstructions.  Further structural and coding-theoretic viewpoints appear
in~\cite{BosekEtAl,Gadouleau,HeLi}.  For ordinary hat guessing on books,
He, Ido, and Przybocki~\cite{HeIdoPrzybocki} developed a coverability
framework and proved that for each fixed \(k\) and all sufficiently large
\(n\),
$
   \HG(B_{k,n})=1+\sum_{m=1}^k m^m.
$
Our coverability language in \cref{sec:coverability} is an adaptation of
that established framework to the proper-coloring constraint.  The new
feature is the support term in \(\abs P+\abs{\supp(P)}\), which records the
colors unavailable to every page on an escaping configuration.
 
The recent proper-coloring paper~\cite{ABBCBKMMRT} proves
\begin{equation}\label{eq:correct-comparison}
   \HGP(G)<\chi(G)\bigl(\HG(G)+1\bigr).
\end{equation}
Together with the ordinary book result, this gives
\[
  \HGP(B_{k,n})<
  (k+1)\left(2+\sum_{m=1}^k m^m\right)
\]
for all sufficiently large \(n\).  In particular, it gives \(<21\) for two
spines.  The factor placement in \cref{eq:correct-comparison} is important;
the term “\(+1\)'' lies inside the factor \(\chi(G)\).
 
In \Cref{sec:preliminaries}, we record the game notation and elementary bounds.
In \Cref{sec:coverability}, we prove the exact guarding characterization and
\cref{thm:intro-stabilization}.  In \cref{sec:two-spine}, we calculate
\(C_2\) explicitly, prove the uniform upper bound to be \(11\), and obtain the eventual matching
lower bound. In \Cref{sec:finite-cases}, we find the exact values at \(n=3,4\)
and the linear upper bound.  In \Cref{sec:affine}, we analyze the affine strategies in detail. In \cref{sec:general-k}, we give the general box bounds. All computational claims and accompanying artifacts are documented in \cref{sec:computation} and Appendix~\ref{app:orbits}.
 
\section{Preliminaries}\label{sec:preliminaries}
 
\subsection{Basics}
 
Write the spine colors as \(\vec{x}=(x_1,\dots,x_k)\) and the page colors
as \(\vec{z}=(z_1,\dots,z_n)\).  Let
\[
  \Omega_{k,q}
  =\set{\vec{x}\in[q]^k:x_i\neq x_j\text{ whenever }i\neq j}.
\]
A proper coloring is a pair
\((\vec{x},\vec{z})\in\Omega_{k,q}\times[q]^n\) satisfying
$
   z_j\notin\set{x_1,\dots,x_k}\
   \text{for}\ 1\leq j\leq n.
$
There is no restriction among the page colors.  Thus, \(B_{k,n}\) has
$
   (q)_k(q-k)^n$ proper \(q\)-colorings, where
   $(q)_k=q(q-1)\cdots(q-k+1).
$
 
Page \(j\) sees the entire spine tuple and uses a function
$
  g_j:\Omega_{k,q}\longrightarrow[q].
$
Spine \(i\) sees the other spine colors and every page color, and uses a
function of \((\vec{x}_{-i},\vec{z})\).  Values of guess functions on
inputs that cannot occur in a proper coloring are immaterial.  We may also
assume that every page guess is \emph{legal} so
$
   g_j(\vec{x})\notin\set{x_1,\dots,x_k}
$
because an illegal guess is never correct and can be replaced arbitrarily.
 
\subsection{Monotonicity}
 
\begin{lemma}[Palette monotonicity]\label{lem:q-monotone}
If a graph has a winning proper strategy with \(q\) colors, then it has one
with every \(q'\leq q\) for which proper \(q'\)-colorings exist.
\end{lemma}
 
\begin{proof}
Restrict every guess function to visible colorings from \([q']\),\ replacing
outputs outside \([q']\) arbitrarily.  Every proper \(q'\)-coloring is also a
proper \(q\)-coloring.  On it, a correct output of the original strategy
already lies in \([q']\) so that correct output is unchanged.
\end{proof}
 
\begin{lemma}[Induced-subgraph monotonicity]\label{lem:induced-monotone}
If \(H\) is an induced subgraph of \(G\), then
\(\HGP(H)\leq\HGP(G)\).  In particular,
\(\HGP(B_{k,n})\) is non-decreasing in \(n\).
\end{lemma}
 
\begin{proof}
Let the vertices of \(H\) use a winning strategy while ignoring their
neighbors in \(G-H\) and let the remaining vertices guess arbitrarily.
The restriction of every proper coloring of \(G\) to \(H\) is proper so a
vertex of \(H\) is correct.
\end{proof}
 
\subsection{General union bound}
 
We will use the book-graph union bound from~\cite{ABBCBKMMRT} to begin with; the following is its short proof.
 
\begin{lemma}[Union bound]\label{lem:union-bound}
Let \(q\geq k+1\).  If
$
   (q-k-1)^n>k(q-k+1)^{n-1},
$
then \(\HGP(B_{k,n})<q\).
\end{lemma}
 
\begin{proof}
Fix a strategy with \(q\) colors.  A fixed spine sees a proper coloring of
a copy of \(B_{k-1,n}\).  It can be correct in at most one extension of each
visible coloring. Hence, on at most
$
  (q)_{k-1}(q-k+1)^n
$
proper colorings.  All \(k\) spines together are therefore correct on at
most \(k(q)_{k-1}(q-k+1)^n\) colorings.
 
For each fixed spine tuple, a page has at least \(q-k-1\) legal colors on
which its guess is wrong.  Hence, at least
\((q)_k(q-k-1)^n\) proper colorings make every page wrong.  If the strategy
wins, all of these must be covered by the spines and so
\[
   (q)_k(q-k-1)^n
   \leq k(q)_{k-1}(q-k+1)^n.
\]
Canceling \((q)_{k-1}(q-k+1)\) gives the contrapositive of the stated
inequality.
\end{proof}
 
For \((k,n,q)=(2,3,8)\), the inequality is
\(5^3>2\cdot7^2\).  Thus,
\begin{equation}\label{eq:b23-upper}
   \HGP(B_{2,3})\leq7.
\end{equation}
However, at \(q=7\), \(4^3<2\cdot6^2\), so a different argument is
needed.
 
\section{Coverability, escape, and stabilization}
\label{sec:coverability}
 
\subsection{Coverability criterion}
 
For \(\vec{x}\in\Omega_{k,q}\), let
\(\pi_i(\vec{x})=\vec{x}_{-i}\); this is the tuple obtained by deleting coordinate
\(i\).
 
\begin{definition}\label{def:coverable}
A set \(A\subseteq\Omega_{k,q}\) is \emph{\(i\)-saturated} if
\(\pi_i\restr A\) is injective.  A configuration
\(P\subseteq\Omega_{k,q}\) is \emph{\(k\)-coverable} if it has a partition
$
   P=P_1\mathbin{\dot\cup}\cdots\mathbin{\dot\cup}P_k
$
for which \(P_i\) is \(i\)-saturated for every \(i\).
\end{definition}
 
The interpretation is immediate. If \(P_i\) is assigned to spine \(i\),
then no two assigned tuples give that spine the same visible spine colors
while requiring different guesses.
 
\begin{lemma}[Hall's form of coverability]\label{lem:hall-coverability}
A finite configuration \(P\) is \(k\)-coverable if and only if
\begin{equation}\label{eq:hall-coverability}
   \abs Q\leq\sum_{i=1}^k\abs{\pi_i(Q)}
   \ \text{for every \(Q\subseteq P\)}.
\end{equation}
\end{lemma}
 
\begin{proof}
Form a bipartite incidence graph.  Its left vertices are the tuples in
\(P\), and its right vertices are the labeled slots
\((i,\vec{u})\) with \(\vec{u}\in\pi_i(P)\).  Join
\(\vec{x}\) to \((i,\pi_i(\vec{x}))\).  A matching saturating \(P\)
assigns each tuple to a distinct slot, which is exactly a coverability
partition.  For \(Q\subseteq P\), its neighborhood has size
\(\sum_i\abs{\pi_i(Q)}\), so Hall's theorem~\cite{Hall} gives precisely
\cref{eq:hall-coverability}.
\end{proof}
 
Fix page functions \(g_1,\dots,g_n\).  For a page vector
\(\vec{z}\in[q]^n\), define the \emph{target configuration}
\begin{equation}\label{eq:general-target}
 T_{\vec{z}}=
 \set{\vec{x}\in\Omega_{k,q}:
   z_j\notin\set{x_1,\dots,x_k}
   \text{ and }g_j(\vec{x})\neq z_j
   \text{ for every \(j\)}}.
\end{equation}
It consists of the proper spine tuples on which every page is wrong.
 
\begin{lemma}[Target criterion]\label{lem:target-criterion}
The fixed page functions extend to a winning strategy on \(B_{k,n}\) if
and only if \(T_{\vec{z}}\) is \(k\)-coverable for every
\(\vec{z}\in[q]^n\).
\end{lemma}
 
\begin{proof}
Fix \(\vec{z}\).  Spine \(i\)'s guess is a function of
\(\vec{x}_{-i}\), so the spine can cover an \(i\)-saturated subset and no
other subset.  If all target tuples are covered by the spines, assigning
each tuple to one correct spine yields a coverability partition.
Conversely, from a partition \(T_{\vec{z}}=\dot\bigcup_iP_i\), define
spine \(i\)'s guess on \((\vec{u},\vec{z})\) to be \(x_i\) for the
unique \(\vec{x}\in P_i\) with \(\pi_i(\vec{x})=\vec{u}\), and define
it arbitrarily when no such tuple exists.  Do this independently for each
\(\vec{z}\).  If no page is correct, the spine tuple lies in
\(T_{\vec{z}}\) and its assigned spine is correct.
\end{proof}
 
\subsection{Escape and guarding}
 
For a nonempty configuration \(P\), define
\[
  \supp(P)=
  \bigcup_{\vec{x}\in P}\set{x_1,\dots,x_k},
  \qquad
  g_j(P)=\set{g_j(\vec{x}):\vec{x}\in P}.
\]
 
\begin{definition}\label{def:guard}
Page \(j\) \emph{guards} \(P\) if
$
   [q]\setminus\supp(P)\subseteq g_j(P).
$
The family of pages guards \(P\) if at least one page guards it.
\end{definition}
 
\begin{lemma}[Escape lemma]\label{lem:escape}
For fixed page functions and a configuration \(P\), the following are
equivalent:
\begin{enumerate}
  \item \(P\subseteq T_{\vec{z}}\) for some \(\vec{z}\in[q]^n\);
  \item no page guards \(P\);
  \item \(g_j(P)\cup\supp(P)\neq[q]\) for every \(j\).
\end{enumerate}
\end{lemma}
 
\begin{proof}
The last two statements are equivalent by definition.  Suppose they hold then
for each page \(j\), choose
$
   z_j\in[q]\setminus\bigl(g_j(P)\cup\supp(P)\bigr).
$
Then \(z_j\) avoids every coordinate of every tuple in \(P\), and it differs
from page \(j\)'s guess on every tuple in \(P\).  Hence,
\(P\subseteq T_{\vec{z}}\).
 
Conversely, if \(P\subseteq T_{\vec{z}}\), then for every \(j\) the color
\(z_j\) belongs neither to \(g_j(P)\) nor to \(\supp(P)\).  Thus,
\(g_j(P)\cup\supp(P)\neq[q]\).
\end{proof}
 
\begin{theorem}[Exact guarding characterization]\label{thm:guard-characterization}
Fixed page functions on \(B_{k,n}\) can be completed by spine functions to
a winning \(q\)-color strategy if and only if they guard every non-coverable
configuration \(P\subseteq\Omega_{k,q}\).
\end{theorem}
 
\begin{proof}
If an unguarded non-coverable \(P\) exists, \cref{lem:escape} places it
inside some target \(T_{\vec{z}}\).  That target is non-coverable, so
\cref{lem:target-criterion} rules out a winning completion.
 
Conversely, if the page functions admit no winning completion, then some
target \(T_{\vec{z}}\) is non-coverable by
\cref{lem:target-criterion}.  It is itself a configuration, and
\cref{lem:escape} shows that no page guards it.
\end{proof}
 
\subsection{Coverability constant}
 
The definition of \(C_k\) in \cref{eq:Ck-intro} ranges over configurations
whose colors may be drawn from any finite set; relabeling colors does not
change coverability.  The minimum exists.  Indeed, take pairwise disjoint
sets \(S_1,\dots,S_k\), each of size \(k+1\), and put
\(P=S_1\times\cdots\times S_k\).  An \(i\)-saturated subset of \(P\) has
size at most \((k+1)^{k-1}\), so \(k\) such subsets cannot cover all
\((k+1)^k\) tuples.  Thus, non-coverable finite configurations exist.
Also, every proper \(k\)-tuple is coverable by itself, so a non-coverable
configuration has at least two tuples and at least \(k\) support colors.
In particular, \(C_k\geq k+2\).
 
\begin{observation}\label{obs:single-guard}
A configuration \(P\subseteq\Omega_{k,q}\) can be guarded by one suitably
chosen legal page function if and only if
$
   \abs P+\abs{\supp(P)}\geq q.
$
\end{observation}
 
\begin{proof}
A guard must realize all \(q-\abs{\supp(P)}\) colors outside the support
among its \(\abs P\) values.  This is possible exactly under the displayed
inequality.  Every color outside \(\supp(P)\) is a legal guess at every
tuple in \(P\), so there is no additional obstruction.
\end{proof}
 
\begin{theorem}[Coverability constant]\label{thm:Ck}
For every \(k\geq2\),
\[
  \sup_{n\geq1}\HGP(B_{k,n})=C_k.
\]
More explicitly, no \(B_{k,n}\) has a winning strategy with \(C_k+1\)
colors, while \(B_{k,n}\) has a winning strategy with \(C_k\) colors for
all sufficiently large \(n\).
\end{theorem}
 
\begin{proof}
Choose a non-coverable configuration \(P\) with
\(\abs P+\abs{\supp(P)}=C_k\). Relabel its support inside a palette of
size \(q=C_k+1\).  Every page takes at most \(\abs P\) values on \(P\), but
guarding \(P\) would require all
$
 q-\abs{\supp(P)}=\abs P+1
$
colors outside its support.  Hence, no page can guard \(P\) regardless of
how many pages there are.  By \cref{thm:guard-characterization}, no
\(B_{k,n}\) has a winning \(q\)-color strategy.  Palette monotonicity gives
\(\HGP(B_{k,n})\leq C_k\).
 
Now set \(q=C_k\).  There are finitely many configurations in
\(\Omega_{k,q}\).  For each non-coverable configuration \(P\), the
definition of \(C_k\) gives
\(\abs P+\abs{\supp(P)}\geq q\). By
\cref{obs:single-guard}, choose one page function \(g_P\) that guards
\(P\).  Taking one dedicated page for every non-coverable \(P\) yields a
finite family that guards them all.  The guarding characterization extends
this family to a winning strategy.  Adding arbitrary further pages
preserves the guards, so the strategy exists for every sufficiently large
\(n\).  Together with the upper bound, this proves the result.
\end{proof}
 
\begin{remark}\label{rem:Nk-crude}
The proof gives the crude bound
$
   N_k\leq 2^{(C_k)_k}
$
since \(\Omega_{k,C_k}\) contains \((C_k)_k\) tuples.  The point of
\cref{thm:Ck} is structural, meaning stabilization follows from a finite invariant. We do not attempt to optimize in this paper, but we note that sharper thresholds require more economical families of guards.
\end{remark}
 
\section{Asymptotic value for two-spine}\label{sec:two-spine}
 
\subsection{Pseudoforests}
 
For \(k=2\), consider a configuration \(P\) as a bipartite graph: an ordered
pair \((x,y)\) is an edge from the row copy of \(x\) to the column copy of
\(y\).  Row \(c\) and column \(c\) are distinct vertices, and the diagonal
edge \((c,c)\) is forbidden.
 
\begin{lemma}[Pseudoforest criterion]\label{lem:pseudoforest}
A two-spine configuration is coverable if and only if it is a
pseudoforest. Equivalently, every connected component contains at most one
cycle.
\end{lemma}
 
\begin{proof}
Assigning an edge to spine \(1\) permits at most one assigned edge per
column, while assigning it to spine \(2\) permits at most one per row.  Thus,
coverability is equivalent to assigning every edge injectively to one of
its endpoints.  By Hall's theorem~\cite{Hall}, this is possible exactly when
\(\abs{E(H)}\leq\abs{V(H)}\) for every subgraph \(H\).  That condition is
equivalent to every component having a cyclomatic number of at most one.
\end{proof}

A \emph{bicyclic configuration} is a connected two-spine configuration
\(P\) with \(e(P)=v(P)+1\).  Every non-pseudoforest configuration contains a bicyclic
subconfiguration. Indeed, some connected component has at least two
edges outside any spanning tree. A spanning tree of that component,
together with any two such edges, is connected and has
$
    (v(P)-1)+2=v(P)+1
$
edges. Consequently, in the definition of \(C_2\), it suffices to
minimize over bicyclic configurations.
 
Denote \(R(P)\) and \(K(P)\) for the row-color and column-color sets respectively,
and let
\[
e(P)=\abs P,\
v(P)=\abs{R(P)}+\abs{K(P)},\ \text{and}\
s(P)=\abs{\supp(P)}=\abs{R(P)\cup K(P)}.
\]
 
\begin{lemma}[Extremal configurations]\label{lem:extremal}
Among all bicyclic configurations,
\[
   \min_P\bigl(e(P)+s(P)\bigr)=11.
\]
Up to relabeling colors and interchanging rows with columns, equality occurs
in exactly two ways:
\begin{enumerate}
  \item \(P=R\times K\cong K_{2,3}\), where
        \(\abs R=2\), \(\abs K=3\), and \(R\cap K=\varnothing\).
  \item \(R=\{a,b,x\}\), \(K=\{a,b,y\}\), where
        \(x\neq y\) and \(x,y\notin\{a,b\}\), and $P=(R\times K) \setminus\{(a,a),(b,b)\}.$
\end{enumerate}
\end{lemma}

\begin{proof}
Let \(r=\abs R\), \(c=\abs K\), and \(t=\abs{R\cap K}\).  Then
$
 v=r+c=s+t\ \text{and}\ e=v+1,
$
so
\begin{equation}\label{eq:es-formula}
   e+s=2s+t+1.
\end{equation}
There are at most \(rc-t\) legal edges between \(R\) and \(K\), because
the \(t\) diagonal pairs indexed by \(R\cap K\) are forbidden. Therefore, a necessary
condition for a bicyclic configuration is
\begin{equation}\label{eq:capacity-bicyclic}
  rc-t\geq r+c+1,\
  \text{or equivalently}\
  (r-1)(c-1)\geq t+2.
\end{equation}
 
We first rule out \(s\leq3\).  If \(t=0\), then
\(r+c=s\leq3\), so \((r-1)(c-1)\leq0<2\).  If \(t=1\), then
\(r+c=s+1\leq4\), so \((r-1)(c-1)\leq1<3\).  If \(t=2\), then
\(r+c\leq5\) with \(r,c\geq2\), so
\((r-1)(c-1)\leq2<4\).  Finally, \(t=3\) forces
\(r=c=s=3\), giving \(4<5\).  Each case contradicts
\cref{eq:capacity-bicyclic}.
 
Suppose \(s=4\).  When \(t=0\), we have \(r+c=4\), and the left side of
\cref{eq:capacity-bicyclic} is at most \(1<2\).  When \(t=1\), we have
\(r+c=5\), and it is at most \(2<3\).  Hence, \(t\geq2\) and so
\cref{eq:es-formula} gives \(e+s\geq11\).  Equality forces \(t=2\) and
\(r+c=6\).  The splits \((2,4)\) and \((4,2)\) violate
\cref{eq:capacity-bicyclic}; hence, \(r=c=3\).  Exactly
\(3\cdot3-2=7\) legal pairs are available, while \(e=7\), so all legal
pairs must occur.  This is configuration~(2).
 
If \(s\geq5\), \cref{eq:es-formula} gives \(e+s\geq11\).  Equality forces
\(s=5\) and \(t=0\).  Now \(r+c=5\), \(e=6\), and only the split
\((r,c)=(2,3)\), up to order, supplies six edges.  All six pairs must be
present, giving configuration~(1).  Both displayed configurations are
connected and bicyclic, so the bound and classification are sharp.
\end{proof}
 
\begin{corollary}\label{cor:C2}
\(C_2=11\).  Hence, \(\HGP(B_{2,n})\leq11\) for every \(n\), with equality
for all sufficiently large \(n\).
\end{corollary}
 
\begin{proof}
Every non-coverable configuration contains a bicyclic subconfiguration,
which cannot have larger \(e+s\). Conversely, every bicyclic configuration
is non-coverable.  Thus, \cref{lem:extremal} computes \(C_2\), and
\cref{thm:Ck} finishes the proof.
\end{proof}
 
\begin{remark}[Obstruction at twelve colors]\label{rem:direct12}
The first extremal configuration makes the uniform upper bound transparent.
Choose five distinct colors \(r_1,r_2,c_1,c_2,c_3\) and the six pairs
\(P=\{r_1,r_2\}\times\{c_1,c_2,c_3\}\).  At \(q=12\), a page has at most
six values on \(P\), but seven colors lie outside its support, so no page
can guard \(P\).  The escape lemma then produces a page vector on which
this \(K_{2,3}\) survives.  Its six edges cannot be covered by two spines,
whose row and column capacities total only \(2+3=5\).
\end{remark}
 
\subsection{Stabilization estimate}
 
\begin{theorem}\label{thm:n0}
There is a winning \(11\)-color strategy on \(B_{2,n}\) whenever
\[
  n\geq
  \left\lceil110(\log2)\,9^7\right\rceil
  =364{,}683{,}163 \approx 4 \times 10^8.
\]
\end{theorem}
 
\begin{proof}
For each page, independently choose every legal value \(g(x,y)\) uniformly
from the nine colors in \([11]\setminus\{x,y\}\), independently over the
\(110\) ordered pairs \(x\neq y\).
 
Fix a bicyclic configuration \(P\), set \(s=s(P)\), and let
\(m=11-s\) be the number of colors outside its support.  By
\cref{lem:extremal}, \(s\geq4\), so \(0\leq m\leq7\), and
\(e(P)\geq m\).  If \(m=0\), every page guards \(P\).  Otherwise, choose
\(m\) distinct edges of \(P\) and assign the \(m\) missing colors to them
in a fixed order.  Each missing color is legal at every edge of \(P\) so
one random page realizes this assignment and therefore guards \(P\) with
probability \(9^{-m}\geq9^{-7}\).
 
There are at most \(2^{110}\) configurations on the \(110\) legal ordered
pairs.  With \(N\) independent pages, the union bound gives
\[
 \Pr(\text{some bicyclic \(P\) is unguarded})
 \leq 2^{110}(1-9^{-7})^N
 \leq 2^{110}e^{-N9^{-7}}.
\]
This is less than \(1\) when \(N>110(\log2)9^7\).  Hence, a family guarding
every bicyclic configuration exists for the asserted integer \(N\).
Guarding every bicyclic configuration is enough, because every
non-coverable two-spine configuration contains one.  Apply
\cref{thm:guard-characterization} to conclude.
\end{proof}
 
\begin{remark}
The bound in \cref{thm:n0} is only an existence estimate.  It
does not suggest that the true stabilization threshold is of comparable
size.
\end{remark}
 
\section{Exact finite cases}
\label{sec:finite-cases}
 
\subsection{Strategy on \(B_{2,3}\)}
 
We identify the seven colors with \(\F_7\).  The three pages use the affine
guess functions
\begin{equation}\label{eq:F7-pages}
   g_i(x,y)=x+\lambda_i(y-x),
   \qquad
   (\lambda_1,\lambda_2,\lambda_3)=(2,3,6).
\end{equation}
All arithmetic in this subsection is in \(\F_7\).
 
\begin{lemma}\label{lem:F7-basic}
When \(x\neq y\), the three guesses in \cref{eq:F7-pages} are legal and
pairwise distinct.  For every \(i\) and every \(\theta\in\F_7\), exactly
six ordered pairs \(x\neq y\) satisfy \(g_i(x,y)=\theta\).
\end{lemma}
 
\begin{proof}
The equality \(g_i(x,y)=x\) would force
\(\lambda_i(y-x)=0\), while \(g_i(x,y)=y\) would force
\((\lambda_i-1)(y-x)=0\).  Both are impossible because
\(\lambda_i\notin\{0,1\}\).  Moreover,
$
  g_i(x,y)-g_j(x,y)=(\lambda_i-\lambda_j)(y-x)\neq0
$
for \(i\neq j\).  Finally, for fixed \(\theta\), the equation
\(g_i(x,y)=\theta\) determines \(y\) uniquely from each \(x\). Among the
seven resulting pairs, only the diagonal pair \(x=y=\theta\) is
excluded.
\end{proof}
 
For \(\vec{z}=(z_1,z_2,z_3)\), the target \(T_{\vec z}\) is the bipartite
graph
\begin{equation}\label{eq:F7-target}
 T_{\vec z}=
 \set{(x,y):x\neq y,\ x,y\notin\{z_1,z_2,z_3\},\
        g_i(x,y)\neq z_i\ (1\leq i\leq3)}.
\end{equation}
A \emph{saturated decomposition} of a target is a partition
$
   T_{\vec z}=\cC_{\vec z}\mathbin{\dot\cup}\cR_{\vec z}
$
such that \(\cC_{\vec z}\) contains at most one edge in each column and
\(\cR_{\vec z}\) contains at most one edge in each row.  This is precisely
a two-spine coverability partition.
 
Consider the affine group
\[
  \AGL(1,\F_7)=
  \set{\sigma_{\alpha,c}:t\mapsto\alpha t+c:
        \alpha\in\F_7^\times,\ c\in\F_7},
\]
which acts diagonally on color triples and on ordered pairs.
 
\begin{lemma}[Equivariance]\label{lem:F7-equivariance}
For every \(\sigma\in\AGL(1,\F_7)\),
\[
  g_i(\sigma x,\sigma y)=\sigma g_i(x,y),
  \qquad
  \sigma(T_{\vec z})=T_{\sigma(\vec z)}.
\]
In other words, affine transport preserves saturated decompositions.
\end{lemma}
 
\begin{proof}
For \(\sigma(t)=\alpha t+c\),
$
  g_i(\alpha x+c,\alpha y+c)
  =\alpha\bigl(x+\lambda_i(y-x)\bigr)+c.
$
The target identity follows because \(\sigma\) is a bijection preserving
all equalities and inequalities in \cref{eq:F7-target}.  It sends rows to
rows and columns to columns bijectively, so saturation is preserved.
\end{proof}
 
\begin{lemma}[Nine orbits]\label{lem:F7-orbits}
The diagonal action of \(\AGL(1,\F_7)\) on \(\F_7^3\) has the nine
orbits listed in \cref{tab:F7-orbits}.
\begin{table}[htbp]
\centering
\begin{tabular}{lc}
\toprule
Orbit Type & Representative \\
\midrule
\(AAA\) & \((0,0,0)\) \\
\(AAB\) & \((0,0,1)\) \\
\(ABA\) & \((0,1,0)\) \\
\(BAA\) & \((1,0,0)\) \\
\(ABC,\ \kappa=2,3,4,5,6\) & \((0,1,\kappa)\) \\
\bottomrule
\end{tabular}
\caption{Orbits of the diagonal action of \(\AGL(1,\F_7)\) on \(\F_7^3\).}
\label{tab:F7-orbits}
\end{table}

The constant orbit has size \(7\). Every other orbit has
size \(42\).
\end{lemma}
 
\begin{proof}
Constant triples form one orbit.  Triples with exactly two distinct entries
split according to the position of the unique entry.  If
\(z_1,z_2,z_3\) are distinct, the affine invariant
\[
  \kappa(\vec z)=\frac{z_3-z_1}{z_2-z_1}
\]
lies in \(\F_7\setminus\{0,1\}\) and completely determines the orbit.
The stabilizer of a constant triple consists of the six scalings about its
value.  Every non-constant triple has trivial stabilizer.  Thus, the orbit
sizes are \(7\) and \(42\). A quick verification gives
\(7+8\cdot42=343=\abs{\F_7^3}\).
\end{proof}

We first look at the constant representative in detail. Fix a non-zero row \(x\). For the multipliers
\(\lambda_i=2,3,6\), the equations \(g_i(x,y)=0\) exclude,
respectively, the columns
\[
    y=4x,\qquad y=3x,\qquad y=2x.
\]
Together with the diagonal \(y=x\), these leave exactly two target edges
in every nonzero row.  The resulting graph is a disjoint union of two
six-cycles.  One saturated decomposition is
\begin{align*}
\cC_{(0,0,0)}
 &=\{(1,5),(2,3),(3,1),(4,6),(5,4),(6,2)\},\\
\cR_{(0,0,0)}
 &=\{(1,6),(2,5),(3,4),(4,3),(5,2),(6,1)\}.
\end{align*}
The remaining representatives' exact edge lists
and decompositions appear in Appendix~\ref{app:orbits}. The relevant sizes are
given in \cref{tab:F7-sizes}.
\begin{table}[htbp]
\centering
{\small
\setlength{\tabcolsep}{3.5pt}
\begin{tabular}{lccccccccc}
\toprule
Orbit&\(AAA\)&\(AAB\)&\(ABA\)&\(BAA\)&\(ABC_2\)&\(ABC_3\)&\(ABC_4\)&\(ABC_5\)&\(ABC_6\)\\
\midrule
\(\abs{T_{\vec z}}\)&12&10&10&10&6&7&7&6&6\\
\(\abs{\cC_{\vec z}}\)&6&5&5&5&2&3&4&4&4\\
\(\abs{\cR_{\vec z}}\)&6&5&5&5&4&4&3&2&2\\
\bottomrule
\end{tabular}
}
\caption{Target and decomposition sizes for the nine orbit representatives.}
\label{tab:F7-sizes}
\end{table}

\pagebreak

Every displayed partition is saturated. Thus, the table is easily checkable to see that each representative's target is a pseudoforest.
 
For completeness, we now turn these local decompositions into explicit
global spine functions.  Given \(\vec z\), choose its representative
\(\vec z_0\) and the canonical affine map \(\sigma_{\vec z}\) in
\cref{tab:canonical-affine}.
\begin{table}[htbp]
\centering
\begin{tabular}{lll}
\toprule
Form of \(\vec z\)&\(\vec z_0\)&\(\sigma_{\vec z}(t)\)\\
\midrule
\((a,a,a)\)&\((0,0,0)\)&\(t+a\)\\
\((a,a,b),\ a\neq b\)&\((0,0,1)\)&\((b-a)t+a\)\\
\((a,b,a),\ a\neq b\)&\((0,1,0)\)&\((b-a)t+a\)\\
\((b,a,a),\ a\neq b\)&\((1,0,0)\)&\((b-a)t+a\)\\
\((a,b,c),\ a,b,c\text{ distinct}\)
 &\((0,1,\kappa),\ \kappa=(c-a)/(b-a)\)&\((b-a)t+a\)\\
\bottomrule
\end{tabular}
\caption{Canonical representatives and affine maps for the nine orbits.}
\label{tab:canonical-affine}
\end{table}

Direct substitution gives
\(\sigma_{\vec z}(\vec z_0)=\vec z\).  Define
$
  \cC_{\vec z}=\sigma_{\vec z}(\cC_{\vec z_0}) \
  \text{and} \
  \cR_{\vec z}=\sigma_{\vec z}(\cR_{\vec z_0}).
$
By \cref{lem:F7-equivariance}, this is a saturated decomposition of
\(T_{\vec z}\).
 
For a column \(y'\), let \(\widehat s_{1,\vec z_0}(y')=x'\) when
\((x',y')\) is the unique edge of \(\cC_{\vec z_0}\) in that column;
define it arbitrarily when the column is unused. Similarly, let
\(\widehat s_{2,\vec z_0}(x')=y'\) read the unique
\(\cR_{\vec z_0}\)-edge in row \(x'\) and arbitrary otherwise. Explicitly, the two spines use
\begin{align}
 s_1(y,\vec z)
   &=\sigma_{\vec z}\!\left(
       \widehat s_{1,\vec z_0}(\sigma_{\vec z}^{-1}(y))\right),
       \label{eq:spine1-explicit}\\
 s_2(x,\vec z)
   &=\sigma_{\vec z}\!\left(
       \widehat s_{2,\vec z_0}(\sigma_{\vec z}^{-1}(x))\right).
       \label{eq:spine2-explicit}
\end{align}
These functions cover \(\cC_{\vec z}\) and \(\cR_{\vec z}\),
respectively.  Thus, the page functions, \cref{tab:canonical-affine}, and
\cref{eq:spine1-explicit,eq:spine2-explicit} specify a deterministic
strategy on every input.
 
\begin{theorem}\label{thm:B23}
\(\HGP(B_{2,3})=7\).
\end{theorem}
 
\begin{proof}
The explicit decomposition above covers every target at \(q=7\), so
\cref{lem:target-criterion} proves the lower bound.  The upper bound is
\cref{eq:b23-upper}.
\end{proof}
 
\begin{remark}\label{rem:F7-stats}
The constant-orbit decomposition is invariant, part by part, under every
nonzero scaling; it is enough to check the generator \(t\mapsto3t\).
The complete strategy wins on all \(5250\) proper colorings.  Exactly \(2562\)
are won by a page, \(1386\) of the remaining colorings by spine \(1\), and
\(1302\) by spine \(2\).  In fact, every one of the
\(5\cdot4\cdot3=60\) ordered triples of distinct multipliers from
\(\F_7\setminus\{0,1\}\) gives a winning page strategy.
\end{remark}
 
\subsection{\texorpdfstring{Linear upper bound and \(B_{2,4}\)}
{Linear upper bound and B(2,4)}}
 
\begin{theorem}\label{thm:linear}
For every \(n\geq4\),
$
   \HGP(B_{2,n})\leq n+3.
$
\end{theorem}
 
\begin{proof}
Let \(q=n+4\). Suppose a winning strategy exists; we seek a contradiction. First, restrict the adversary to the
$
   q(q-1)(q-2)
$
colorings, where$\
   (x,y,c,c,\dots,c),
   \ x\neq y,\ \text{and}\ c\notin\{x,y\}.$
Each of the \(n\) pages is correct on at most \(q(q-1)\) of these colorings.
For each \((x,y)\), its one guess agrees with at most one admissible \(c\).
Spine \(1\) also has at most \(q(q-1)\) successes, because its visible input
is determined by \((y,c)\) and at most one hidden \(x\) can equal its
guess. By symmetry, the same holds for spine \(2\).
 
There are \(n+2=q-2\) players, so the sum of their success counts is at
most
\[
  (n+2)q(q-1)=q(q-1)(q-2),
\]
which coincides with the number of restricted colorings.  A winning strategy has at least one
success on every coloring, so equality holds throughout.  Consequently,
every restricted coloring has exactly one correct player, and every player
attains its individual upper bound.
 
It follows that for each fixed \((x,y)\), the page guesses
$
   g_1(x,y),\dots,g_n(x,y)
$
are pairwise distinct elements of
\(L=[q]\setminus\{x,y\}\).  Indeed, each page must attain its per-input
maximum, so its guess is legal; two equal page guesses would make two pages
correct on the corresponding uniform coloring.
 
Partition the pages into sets \(S_1,S_2\) of sizes
\[
  n_1=\lfloor n/2\rfloor,\qquad n_2=\lceil n/2\rceil.
\]
For distinct colors \(a,b\), color every page in \(S_1\) by \(a\) and every
page in \(S_2\) by \(b\).  Fix \((x,y)\) disjoint from \(\{a,b\}\).
Let \(H_i\subseteq L\) be the page guesses from \(S_i\).  The sets
\(H_1,H_2\) are disjoint, have sizes \(n_1,n_2\), and leave a two-element
set \(M=L\setminus(H_1\cup H_2)\).  As \((a,b)\) ranges over the ordered
distinct pairs in \(L\), all pages fail exactly when
\[
  a\in H_2\cup M,\qquad b\in H_1\cup M,\qquad a\neq b.
\]
The number of such pairs is
\begin{equation}\label{eq:linear-failure-count}
  (n_2+2)(n_1+2)-2.
\end{equation}
The subtraction removes the two diagonal choices from \(M\times M\).
 
Double-count the incidences between ordered spine pairs \((x,y)\) and
ordered page colors \((a,b)\), all four colors being appropriately
disjoint, for which every page fails.  Summing
\cref{eq:linear-failure-count} over \((x,y)\) gives
\[
  q(q-1)\bigl((n_1+2)(n_2+2)-2\bigr)
\]
incidences.  For fixed \((a,b)\), spine \(1\) can be correct on at most one
pair per one of the \(q-2\) available columns, and spine \(2\) on at most
one per available row.  Since every page-failure coloring must be covered
by a spine, summing this capacity over \((a,b)\) gives
\[
 q(q-1)\bigl((n_1+2)(n_2+2)-2\bigr)
 \leq q(q-1)\,2(q-2).
\]
Canceling \(q(q-1)\), using \(q-2=n+2\), and \(n_1+n_2=n\) yields
\[
   n_1n_2+2n+2\leq2n+4
   \Longrightarrow n_1n_2\leq2.
\]
However, \(n\geq4\) gives \(n_1,n_2\geq2\), so a contradiction arises.  Thus, no strategy
wins with \(q=n+4\), and \cref{lem:q-monotone} proves the claim.
\end{proof}
 
\begin{corollary}\label{cor:B24}
\(\HGP(B_{2,4})=7\).
\end{corollary}
 
\begin{proof}
\Cref{thm:linear} gives the upper bound \(7\), while
\cref{lem:induced-monotone,thm:B23} give the matching lower bound.
\end{proof}
 
\section{Affine ceiling}\label{sec:affine}
 
\subsection{Construction}
 
Let \(q\) be a prime power and identify the palette with \(\F_q\).  For a
set \(S\subseteq\F_q\setminus\{0,1\}\) of distinct multipliers, assign one
page to each \(\lambda\in S\) and let it guess
\begin{equation}\label{eq:affine-general}
   g_\lambda(x,y)=x+\lambda(y-x).
\end{equation}
The legality, distinctness, and affine equivariance arguments from
\cref{lem:F7-basic,lem:F7-equivariance} hold verbatim.
 
\begin{lemma}[Affine monotonicity]\label{lem:affine-monotone}
If the affine strategy for \(S\) wins and
\(\lambda\notin S\cup\{0,1\}\), then the strategy for
\(S\cup\{\lambda\}\) wins.  Thus, within the admissible range
\(\abs S\leq q-2\), a failure of the maximal multiplier set
\(\F_q\setminus\{0,1\}\) rules out every affine strategy of the form
\cref{eq:affine-general}.
\end{lemma}
 
\begin{proof}
For a page vector \((\vec z,z_{\lambda})\), the new target is a subgraph
of the old target \(T_{\vec z}\). It satisfies all old constraints and
one additional page constraint, and its available row and column colors
can only decrease.  A subgraph of a pseudoforest is a pseudoforest.
\end{proof}
 
\begin{theorem}[\(\mathrm{GF}(8)\) strategy]\label{thm:GF8}
On \(B_{2,6}\), the affine strategy over \(\mathrm{GF}(8)\) with all six
multipliers in \(\mathrm{GF}(8)\setminus\{0,1\}\) is winning.  Consequently,
$
   \HGP(B_{2,n})\geq8\ \text{for every }n\geq6.
$
\end{theorem}
 
\begin{proof}
This proof uses an exhaustive
computation.  For each of the \(8^6=262{,}144\) page vectors, the verifier
constructs the target graph and checks every component has at most as many
edges as vertices.  No target fails.  The field arithmetic, enumeration
order, pseudoforest test, and independent cross-check are described in
\cref{sec:computation}.  The conclusion for larger \(n\)
follows from \cref{lem:induced-monotone}.
\end{proof}
 
\begin{remark}\label{rem:GF8-five}
The threshold of six pages is sharp within this construction.  There are
six five-element subsets of
\(\mathrm{GF}(8)\setminus\{0,1\}\), with each failing on exactly \(224\) of the
\(8^5\) page vectors.
\end{remark}
 
\subsection{Two-block target}
 
Use the maximal multiplier set
\(\Lambda=\F_q\setminus\{0,1\}\) and partition it as
\(\Lambda=A\mathbin{\dot\cup}B\).  Give page color \(0\) to the
multipliers in \(A\) and page color \(1\) to those in \(B\).
 
\begin{lemma}[Two-block target size]\label{lem:two-block}
The resulting target has
$
   \abs{E(T)}=\abs A\,\abs B.
$
The count depends only on the two block sizes.
\end{lemma}
 
\begin{proof}
The available spine colors are
\(L=\F_q\setminus\{0,1\}=\Lambda\).  For \(x,y\in L\), \(x\neq y\), the
unique multipliers for which \cref{eq:affine-general} equals \(0\) and
\(1\) are
\[
   \mu_0(x,y)=\frac{x}{x-y},
   \qquad
   \mu_1(x,y)=\frac{x-1}{x-y}.
\]
Both belong to \(\Lambda\).  The pages all fail precisely when
\(\mu_0\notin A\) and \(\mu_1\notin B\), or equivalently
\((\mu_0,\mu_1)\in B\times A\).
 
The map \((x,y)\mapsto(\mu_0,\mu_1)\) is a bijection from the ordered
distinct pairs in \(L\) to the ordered distinct pairs in \(\Lambda\).  Its
inverse is
\[
   x=\frac{\mu_0}{\mu_0-\mu_1},
   \qquad
   y=\frac{\mu_0-1}{\mu_0-\mu_1}.
\]
Since \(A\cap B=\varnothing\), all pairs in \(B\times A\) are distinct,
and the target has the asserted size.
\end{proof}
 
\begin{theorem}[Affine ceiling]\label{thm:affine-ceiling}
For every prime power \(q\geq9\), no affine strategy of the form
\cref{eq:affine-general} wins on \(B_{2,n}\) for any admissible number of
distinct multipliers \(n\leq q-2\).
\end{theorem}
 
\begin{proof}
By \cref{lem:affine-monotone}, it suffices to defeat the maximal strategy.
For \(q\geq11\), split the \(q-2\) multipliers as evenly as possible.
\Cref{lem:two-block} gives
\[
  \abs{E(T)}
  =\left\lfloor\frac{(q-2)^2}{4}\right\rfloor
  >2(q-2)
\]
because \(q-2>8\).  The target has at most \(2(q-2)\) row and column
vertices, so it is not a pseudoforest.
 
It remains to treat \(q=9\).  Write
\(\F_9=\F_3[\xi]/(\xi^2+1)\) and fix throughout the remainder of this proof
the multiplier order
\begin{equation}\label{eq:F9-order}
 (\lambda_1,\dots,\lambda_7)
   =(\xi,\ 1+\xi,\ 2\xi,\ 1+2\xi,\ 2,\ 2+\xi,\ 2+2\xi),
\end{equation}
so that page \(i\) guesses \(g_{\lambda_i}\).  Take the page vector
$
   \vec z=(0,0,0,0,1,1,2)
$
relative to \cref{eq:F9-order}.  A direct substitution in
\cref{eq:affine-general} gives \(\abs{T_{\vec z}}=10\); the target is the
disjoint union of the isolated edge \((\xi,2\xi)\) and the connected
subgraph \(D\) on the row and column copies of
\[
  1+\xi,\quad1+2\xi,\quad2+\xi,\quad2+2\xi
\]
whose nine edges are
\begin{multline}\label{eq:F9-certificate}
 \bigl(\{1+\xi,1+2\xi\}\times\{2+\xi,2+2\xi\}\bigr)
 \ \cup\
 \bigl(\{2+\xi,2+2\xi\}\times\{1+\xi,1+2\xi\}\bigr)
 \ \cup\ \{(2+2\xi,2+\xi)\}.
\end{multline}
We stress that \cref{eq:F9-certificate} displays the component \(D\), not
the whole target. \(D\) is a proper subgraph of \(T_{\vec z}\), which has
one further edge.  The certificate does not depend on that extra edge.
Indeed, \(D\) is connected with eight vertices and nine edges, so \(D\) is
not a pseudoforest, and hence neither is \(T_{\vec z}\).  Each of the ten
memberships in the target is a direct substitution in
\cref{eq:affine-general}, so the whole computation is checkable by hand.
 
Permuting the assignment of multipliers to pages merely permutes the
coordinates of a defeating page vector, so the obstruction applies to
every multiplier ordering.  This completes the maximal case, and
\cref{lem:affine-monotone} rules out all subsets.
\end{proof}
 
\begin{remark}\label{rem:F9-count}
Among the \(105\) distinct assignments of the multiset with multiplicities
$
  0^4,\ 1^2,\ \text{and }2^1
$
to the multiplier order \cref{eq:F9-order}, exactly three produce a
non-pseudoforest.  A separate full enumeration finds \(8568\) failing page
vectors among all \(9^7=4{,}782{,}969\) vectors.  That enumeration lists the
multipliers in the ascending order of their integer encodings rather than in
the order \cref{eq:F9-order}, so it reports a different first failure; the
two orders are related by a permutation of the page coordinates and defeat
the same strategy.  These counts are not assumptions in the proof, but are
verified in \cref{sec:computation}.
\end{remark}
 
\section{General spine size}\label{sec:general-k}
 
\subsection{Box obstructions}
 
\begin{theorem}\label{thm:box}
Let \(k\geq2\) and let integers \(m_1,\dots,m_k\geq2\) satisfy
$
   \sum_{i=1}^k\frac1{m_i}<1.
$
Then
\[
   C_k\leq \prod_{i=1}^k m_i+\sum_{i=1}^k m_i.
\]
Consequently, the same expression is an upper bound for
\(\HGP(B_{k,n})\) for every \(n\).
\end{theorem}
 
\begin{proof}
Choose pairwise disjoint color sets \(S_1,\dots,S_k\) with
\(\abs{S_i}=m_i\) and consider the box
\[
   P=S_1\times\cdots\times S_k.
\]
Every tuple in \(P\) is proper.  An \(i\)-saturated subset has at most
\(\prod_{j\neq i}m_j\) elements, one for each visible
\((k-1)\)-tuple.  Therefore, \(k\) saturated parts can cover at most
\[
   \sum_{i=1}^k\prod_{j\neq i}m_j
   =\left(\prod_{i=1}^k m_i\right)
      \left(\sum_{i=1}^k\frac1{m_i}\right)
   <\prod_{i=1}^k m_i=\abs P
\]
tuples.  Hence, \(P\) is non-coverable.  Its support has size
\(\sum_i m_i\) so the definition of \(C_k\) gives the bound.
\Cref{thm:Ck} then gives the uniform bound on \(\HGP\).
\end{proof}
 
For small \(k\), minimizing the box expression gives \cref{tab:box-bounds}.
\begin{table}[htbp]
\centering
\begin{tabular}{ccll}
\toprule
\(k\)&Optimal side lengths&Box bound on \(C_k\)
  &Previous bound\\
\midrule
2&\((2,3)\)&11&\(<21\)\\
3&\((3,3,4)\)&46&\(<136\)\\
4&\((3,4,5,5)\)&317&\(<1450\)\\
5&\((4,5,5,5,7)\)&3526&\(<20490\)\\
\bottomrule
\end{tabular}
\caption{Optimized box bounds for general spine sizes.}
\label{tab:box-bounds}
\end{table}

Each listed optimum is unique up to permuting the side lengths.  The search
is finite without imposing an arbitrary cutoff.  Once an admissible tuple
of objective value \(M\) is known, any improving tuple has
\(\prod_jm_j<M\). Since the other \(k-1\) sides are at least \(2\),
$
  m_i<\frac{M}{2^{k-1}}
$
for every \(i\).  Exhaustive search within these caps
proves the first four optima in \cref{tab:box-bounds}.

\section{Computational verification}\label{sec:computation}
 
All computations are exhaustive and deterministic. All Python programs use
only the standard library. All C programs conform to C11.  A one-command
driver compiles with
\[
\texttt{gcc -std=c11 -O2 -Wall -Wextra -Wpedantic}
\]
and checks every headline output.  No random seeds, external data, or
configuration files are used. The complete verification software is provided as supplementary material.

\
 
The analytic arguments prove the coverability characterization,
stabilization theorem, extremal value \(C_2=11\), linear bound, two-block
lemma, affine obstruction, and general box theorem.

Several statements below are supported by exhaustive enumeration, and it is
worth separating those that carry logical weight from those that are merely
informative.  Only \cref{thm:GF8} is both load-bearing and beyond hand
verification: its proof requires that all \(8^6=262{,}144\) targets be
pseudoforests, and no shorter argument is offered here.  Every other
enumeration either supports a claim that can also be checked by hand, or
reports a count on which no proof depends.  Specifically, the finite orbit
tables for \(B_{2,3}\) underlying \cref{thm:B23} are printed in
Appendix~\ref{app:orbits} and can be checked directly from there. The \(q=9\) case of
\cref{thm:affine-ceiling} rests on the ten target memberships of
\cref{eq:F9-certificate}, each a single substitution. The counts in
\cref{rem:GF8-five,rem:F9-count} are reported
for completeness and are used nowhere in any proof.  The programs in
\cref{tab:verification-programs} recompute all of these.

{\small
\begin{longtable}{@{}
  >{\ttfamily\raggedright\arraybackslash}p{0.29\textwidth}
  >{\raggedright\arraybackslash}p{0.67\textwidth}@{}}
\toprule
\textnormal{File}&\textnormal{Purpose}\\
\midrule
\endfirsthead
\toprule
\textnormal{File}&\textnormal{Purpose (continued)}\\
\midrule
\endhead
\bottomrule
\endfoot
\bottomrule
\noalign{\vskip\abovecaptionskip}
\caption{Verification programs.}
\label{tab:verification-programs}\\
\endlastfoot
verify\_b23.py&
Recomputes the nine \(\AGL(1,\F_7)\)-orbit targets, validates every
decomposition, checks canonical transport and stabilizer invariance, plays
all \(5250\) proper colorings, reproduces the win census, and tests all
\(60\) ordered multiplier triples.\\
verify\_b24\_counts.py&
Checks the finite counting identities underlying the \(n=4\) specialization
of \cref{thm:linear}, including all \(90\) ordered disjoint pairs of
two-subsets of a six-set.\\
enum\_configs.py&
Independently enumerates the extremal two-spine configurations and verifies
the minimum and equality classification in \cref{lem:extremal}.\\
cert\_q9.py&
Reconstructs \cref{eq:F9-certificate} edge by edge, checks the three
successful attachments among \(105\), and validates an independent first
failure from the full enumerator.\\
boxopt.py&
Performs the capped exact-rational search giving \cref{tab:box-bounds} and
checks uniqueness up to order.\\
common.h&
Implements the documented \(\F_7\), \(\mathrm{GF}(8)\), and
\(\mathrm{GF}(9)\) arithmetic and a union--find pseudoforest test.\\
verify\_affine.c&
Enumerates every page vector for the affine results in
\cref{thm:GF8,rem:GF8-five,thm:affine-ceiling}.\\
checkfull.c&
General target verifier for an arbitrary strategy table.  Its
component-by-component depth-first pseudoforest test shares no code with
\texttt{common.h}.\\
% run\_all.sh&
% Compiles both C programs, runs the complete suite, compares \(23\) expected
% outputs, and exits nonzero on any mismatch.\\
% README.txt&
% Documents dependencies, field encodings, commands, input formats, and
% expected outputs.\\
% LICENSE&
% MIT license for the software artifact.\\
\end{longtable}
}
 
% \subsection{Independent checks and exact totals}
 
% The finite-field encodings are fixed as follows.
% \begin{itemize}
%   \item \(\F_7\): integers \(0,\dots,6\), with arithmetic modulo \(7\).
%   \item \(\mathrm{GF}(8)\): binary polynomials modulo
%         \(X^3+X+1\).
%   \item \(\mathrm{GF}(9)\): the integer \(3a+b\) represents
%         \(a+bX\) in \(\F_3[X]/(X^2+1)\).  In this encoding the field
%         element \(1\) is the integer \(3\), and the nonzero, nonunit
%         multipliers are encoded by
%         \(\{1,2,4,5,6,7,8\}\).
% \end{itemize}
 
% The exhaustive affine runs produce the following exact totals.
% \begin{center}
% \small
% \begin{tabular}{cccc}
% \toprule
% \(q\)&pages&page vectors&bad targets\\
% \midrule
% \(7\)&\(3\)&\(343\)&\(0\)\\
% \(8\)&\(6\)&\(262{,}144\)&\(0\)\\
% \(8\)&\(5\)&\(32{,}768\)&\(224^\ast\)\\
% \(9\)&\(7\)&\(4{,}782{,}969\)&\(8568\)\\
% \bottomrule
% \end{tabular}
% \end{center}
% Here \({}^\ast\) means \(224\) bad targets for each of the six
% five-multiplier subsets.
% For the five-page \(\mathrm{GF}(8)\) example, \texttt{run\_all.sh}
% generates the entire strategy table independently and passes it to
% \texttt{checkfull.c}; that separate implementation again returns exactly
% \(224\) bad targets.  The driver also verifies that malformed input is
% rejected with the documented exit status.

\section{Conclusion}

\subsection{Current status and conjecture}

Combining all preceding results with the values for \(n=1,2\)
from~\cite{ABBCBKMMRT} gives the current picture of two-spine book graphs,
summarized in \cref{tab:status}. Note that
\(n_0\) here denotes the least stabilization threshold, not the crude upper
bound in \cref{thm:n0}.

\begin{table}[htbp]
\centering
{\small
\begin{tabular}{ccl}
\toprule
\(n\) & Updated range & Source \\
\midrule
\(1\) & \(5\) & \(B_{2,1}=K_3\); \cite{ABBCBKMMRT} \\
\(2\) & \(6\) & \(B_{2,2}=K_4-e\); \cite{ABBCBKMMRT} \\
\(3\) & \(7\) & \Cref{thm:B23} \\
\(4\) & \(7\) & \Cref{cor:B24} \\
\(5\) & \(\{7,8\}\) & Monotonicity; \cref{thm:linear} \\
\(6\) & \(\{8,9\}\) & \Cref{thm:GF8,thm:linear} \\
\(7\) & \(\{8,9,10\}\) & \Cref{thm:GF8,thm:linear} \\
\(8\leq n<n_0\) & \(\{8,9,10,11\}\)
    & \Cref{thm:GF8}; \cref{cor:C2} \\
\(n\geq n_0\) & \(11\)
    & \Cref{thm:n0}; \cref{cor:C2} \\
\bottomrule
\end{tabular}
}
\caption{Current status of two-spine book graphs.}
\label{tab:status}
\end{table}
 
\begin{conjecture}\label{conj:finite}
For every \(n\geq4\),
$
   \HGP(B_{2,n})=\min\{n+3,11\}.
$
\end{conjecture}
 
The upper bound in \cref{conj:finite} is proved.  Its lower bound remains
open for the finite range before stabilization.

\subsection{Open extremal questions}
 
Taking \(m_1=\cdots=m_k=k+1\) in \cref{thm:box} yields
\begin{equation}\label{eq:uniform-k}
   C_k\leq(k+1)^k+k(k+1)
        =\exp\bigl((1+o(1))k\log k\bigr).
\end{equation}
The bound from \cref{eq:correct-comparison} and the ordinary book theorem is
\[
  (k+1)\left(2+\sum_{m=1}^k m^m\right),
\]
which is larger by a factor of order \(k\).  Unlike that comparison, the
box bound applies to every \(n\), not only to sufficiently large books.
 
For \(k=2\), the box \((2,3)\) is the first extremal configuration in
\cref{lem:extremal} and attains \(C_2\).  For \(k\geq3\), the
stabilization theorem identifies the exact asymptotic value as \(C_k\), but
determining \(C_k\), and deciding whether a box is always extremal, remains
open.

\appendix
\numberwithin{table}{section}
 
\section{Orbit tables for \texorpdfstring{$B_{2,3}$}{B2,3} at \texorpdfstring{$q=7$}{q=7}}\label{app:orbits}
 
For each representative \(\vec z\), this appendix gives the target
adjacency list and a saturated decomposition
\(T_{\vec z}=\cC_{\vec z}\dot\cup\cR_{\vec z}\).
The set \(\cC\) has at most one edge in each column, and \(\cR\) has at most
one edge in each row.  Every entry is recomputed by
\texttt{verify\_b23.py}.
The target-column data are listed in
\cref{tab:orbit-AAA,tab:orbit-AAB,tab:orbit-ABA,tab:orbit-BAA,%
tab:orbit-ABC2,tab:orbit-ABC3,tab:orbit-ABC4,tab:orbit-ABC5,%
tab:orbit-ABC6}.
 
\subsection*{\texorpdfstring{\(AAA\): \(\vec z=(0,0,0)\), \(\abs T=12\)}
{AAA: z=(0,0,0), |T|=12}}
 
\begin{table}[htbp]
\centering
{\small
\begin{tabular}{c@{\qquad}l}
\toprule
Row \(x\)&Target columns\\
\midrule
1&\(\{5,6\}\)\\
2&\(\{3,5\}\)\\
3&\(\{1,4\}\)\\
4&\(\{3,6\}\)\\
5&\(\{2,4\}\)\\
6&\(\{1,2\}\)\\
\bottomrule
\end{tabular}
}
\caption{Target columns for the \(AAA\) representative.}
\label{tab:orbit-AAA}
\end{table}
\begin{align*}
\cC&=\{(1,5),(2,3),(3,1),(4,6),(5,4),(6,2)\},\\
\cR&=\{(1,6),(2,5),(3,4),(4,3),(5,2),(6,1)\}.
\end{align*}

\pagebreak
 
\subsection*{\texorpdfstring{\(AAB\): \(\vec z=(0,0,1)\), \(\abs T=10\)}
{AAB: z=(0,0,1), |T|=10}}
 
\begin{table}[htbp]
\centering
{\small
\begin{tabular}{c@{\qquad}l}
\toprule
Row \(x\)&Target columns\\
\midrule
2&\(\{4,5\}\)\\
3&\(\{4,6\}\)\\
4&\(\{3,6\}\)\\
5&\(\{3,4\}\)\\
6&\(\{2,5\}\)\\
\bottomrule
\end{tabular}
}
\caption{Target columns for the \(AAB\) representative.}
\label{tab:orbit-AAB}
\end{table}
\begin{align*}
\cC&=\{(2,5),(3,6),(4,3),(5,4),(6,2)\},\\
\cR&=\{(2,4),(3,4),(4,6),(5,3),(6,5)\}.
\end{align*}
 
\subsection*{\texorpdfstring{\(ABA\): \(\vec z=(0,1,0)\), \(\abs T=10\)}
{ABA: z=(0,1,0), |T|=10}}
 
\begin{table}[htbp]
\centering
{\small
\begin{tabular}{c@{\qquad}l}
\toprule
Row \(x\)&Target columns\\
\midrule
2&\(\{3,5,6\}\)\\
3&\(\{2,4\}\)\\
4&\(\{5,6\}\)\\
5&\(\{2,4\}\)\\
6&\(\{4\}\)\\
\bottomrule
\end{tabular}
}
\caption{Target columns for the \(ABA\) representative.}
\label{tab:orbit-ABA}
\end{table}
\begin{align*}
\cC&=\{(2,3),(2,6),(3,4),(4,5),(5,2)\},\\
\cR&=\{(2,5),(3,2),(4,6),(5,4),(6,4)\}.
\end{align*}
 
\subsection*{\texorpdfstring{\(BAA\): \(\vec z=(1,0,0)\), \(\abs T=10\)}
{BAA: z=(1,0,0), |T|=10}}
 
\begin{table}[htbp]
\centering
{\small
\begin{tabular}{c@{\qquad}l}
\toprule
Row \(x\)&Target columns\\
\midrule
2&\(\{3\}\)\\
3&\(\{4,5\}\)\\
4&\(\{2,3\}\)\\
5&\(\{2,4,6\}\)\\
6&\(\{2,3\}\)\\
\bottomrule
\end{tabular}
}
\caption{Target columns for the \(BAA\) representative.}
\label{tab:orbit-BAA}
\end{table}
\begin{align*}
\cC&=\{(3,5),(4,3),(5,4),(5,6),(6,2)\},\\
\cR&=\{(2,3),(3,4),(4,2),(5,2),(6,3)\}.
\end{align*}

\
 
\subsection*{\texorpdfstring{\(ABC_2\): \(\vec z=(0,1,2)\), \(\abs T=6\)}
{ABC2: z=(0,1,2), |T|=6}}
 
\begin{table}[htbp]
\centering
{\small
\begin{tabular}{c@{\qquad}l}
\toprule
Row \(x\)&Target columns\\
\midrule
3&\(\{6\}\)\\
4&\(\{5\}\)\\
5&\(\{3,4\}\)\\
6&\(\{4,5\}\)\\
\bottomrule
\end{tabular}
}
\caption{Target columns for the \(ABC_2\) representative.}
\label{tab:orbit-ABC2}
\end{table}
\begin{align*}
\cC&=\{(5,4),(6,5)\},\\
\cR&=\{(3,6),(4,5),(5,3),(6,4)\}.
\end{align*}
 
\subsection*{\texorpdfstring{\(ABC_3\): \(\vec z=(0,1,3)\), \(\abs T=7\)}
{ABC3: z=(0,1,3), |T|=7}}
 
\begin{table}[htbp]
\centering
{\small
\begin{tabular}{c@{\qquad}l}
\toprule
Row \(x\)&Target columns\\
\midrule
2&\(\{5,6\}\)\\
4&\(\{6\}\)\\
5&\(\{2,4\}\)\\
6&\(\{4,5\}\)\\
\bottomrule
\end{tabular}
}
\caption{Target columns for the \(ABC_3\) representative.}
\label{tab:orbit-ABC3}
\end{table}
\begin{align*}
\cC&=\{(2,6),(5,4),(6,5)\},\\
\cR&=\{(2,5),(4,6),(5,2),(6,4)\}.
\end{align*}
 
\subsection*{\texorpdfstring{\(ABC_4\): \(\vec z=(0,1,4)\), \(\abs T=7\)}
{ABC4: z=(0,1,4), |T|=7}}
 
\begin{table}[htbp]
\centering
{\small
\begin{tabular}{c@{\qquad}l}
\toprule
Row \(x\)&Target columns\\
\midrule
2&\(\{3,5,6\}\)\\
3&\(\{6\}\)\\
5&\(\{2,3\}\)\\
6&\(\{5\}\)\\
\bottomrule
\end{tabular}
}
\caption{Target columns for the \(ABC_4\) representative.}
\label{tab:orbit-ABC4}
\end{table}
\begin{align*}
\cC&=\{(2,3),(2,5),(2,6),(5,2)\},\\
\cR&=\{(3,6),(5,3),(6,5)\}.
\end{align*}

\pagebreak
 
\subsection*{\texorpdfstring{\(ABC_5\): \(\vec z=(0,1,5)\), \(\abs T=6\)}
{ABC5: z=(0,1,5), |T|=6}}
 
\begin{table}[htbp]
\centering
{\small
\begin{tabular}{c@{\qquad}l}
\toprule
Row \(x\)&Target columns\\
\midrule
2&\(\{3\}\)\\
3&\(\{2,4,6\}\)\\
4&\(\{6\}\)\\
6&\(\{4\}\)\\
\bottomrule
\end{tabular}
}

\caption{Target columns for the \(ABC_5\) representative.}
\label{tab:orbit-ABC5}
\end{table}
\begin{align*}
\cC&=\{(2,3),(3,2),(3,4),(3,6)\},\\
\cR&=\{(4,6),(6,4)\}.
\end{align*}
 
\subsection*{\texorpdfstring{\(ABC_6\): \(\vec z=(0,1,6)\), \(\abs T=6\)}
{ABC6: z=(0,1,6), |T|=6}}
 
\begin{table}[htbp]
\centering
{\small
\begin{tabular}{c@{\qquad}l}
\toprule
Row \(x\)&Target columns\\
\midrule
2&\(\{3\}\)\\
3&\(\{2,4\}\)\\
4&\(\{5\}\)\\
5&\(\{2,3\}\)\\
\bottomrule
\end{tabular}
}
\caption{Target columns for the \(ABC_6\) representative.}
\label{tab:orbit-ABC6}
\end{table}
\begin{align*}
\cC&=\{(2,3),(3,4),(4,5),(5,2)\},\\
\cR&=\{(3,2),(5,3)\}.
\end{align*}

\section*{Funding}

This research did not receive any specific grant from funding
agencies in the public, commercial, or not-for-profit sectors.

\section*{Conflict of interest}
 
The author declares no competing interests.
 
% \section*{Acknowledgments}
 
% The author thanks Anna G for her helpful comments and support.
 
% \section*{Data and software availability}
 
% The complete deterministic software artifact is distributed with this
% manuscript.  It includes source code, a one-command test driver, expected
% outputs, documentation, and an open-source license.
 
% \section{Artifact commands and expected results}\label{app:artifact}
 
% From the artifact's \texttt{code/} directory, run
% \begin{verbatim}
% sh run_all.sh
% \end{verbatim}
% The final line is
% \begin{verbatim}
% ALL 23 EXPECTED-OUTPUT CHECKS PASSED
% \end{verbatim}
% The complete run takes roughly one minute on a current laptop and is
% dominated by the \(9^7\) affine enumeration.
 
% The general verifier is invoked as
% \begin{verbatim}
% ./checkfull q n strategyfile maxreport
% \end{verbatim}
% The strategy file contains \(q^2n\) whitespace-separated integers
% \(g_j(x,y)\), with page index outermost, then row and column.  Diagonal
% entries are ignored and are conventionally written as \(-1\).
% \texttt{checkfull} prints \texttt{WIN} when every target is a pseudoforest,
% and otherwise prints the first requested failures followed by
% \texttt{TOTALBAD} and the exact count.

\section*{Declaration of generative AI and AI-assisted technologies
in the manuscript preparation process}

During the preparation of this work, the author used
OpenAI's GPT-5.6 Sol for literature searching, language editing, supplementary code organization, and LaTeX formatting.
After using this tool, the author reviewed and edited the
content as needed and takes full responsibility for the content of
the published article.

\end{document}